\documentclass[12pt]{article}
\usepackage{amsmath}
\usepackage{amssymb}
\usepackage{url, enumerate}

\newcommand{\eop}{\bigstar}  

\newcommand{\cf}{{\rm cf}}

\newenvironment{proof}{\noindent{\bf Proof.}}{\par\bigskip}

\newtheorem{THEOREM}{Theorem}[section]

\newtheorem{Conclusion}[THEOREM]{Conclusion}

\newtheorem{Hypothesis}[THEOREM]{Hypothesis}

\newtheorem{LEMMA}[THEOREM]{Lemma}

\newtheorem{Main Theorem}[THEOREM]{Main Theorem}
\newenvironment{main Theorem}{\begin{Main Theorem}} 
{\end{Main Theorem}}

\newtheorem{Theorem}[THEOREM]{Theorem}
\newenvironment{theorem}{\begin{Theorem}}{\end{Theorem}}

\newtheorem{Definition}[THEOREM]{Definition}

\newtheorem{Conventions}[THEOREM]{Conventions}

\newtheorem{Main Definition}[THEOREM]{Main Definition}
\newenvironment{main definition}{\begin{Main Definition}}
{\end{Main Definition}}

\newtheorem{Lemma}[THEOREM]{Lemma}

\newtheorem{Notation}[THEOREM]{Notation}

\newtheorem{Convention}[THEOREM]{Convention}

\newtheorem{Note}[THEOREM]{Note}

\newtheorem{Observation}[THEOREM]{Observation}

\newtheorem{Remark}[THEOREM]{Remark}

\newtheorem{Question}[THEOREM]{Question}

\newtheorem{Main Fact}[THEOREM]{Main Fact}
\newenvironment{main Fact}{\begin{Main Fact}}{\end{Main Fact}}

\newtheorem{Fact}[THEOREM]{Fact}

\newtheorem{Subfact}[THEOREM]{Subfact}

\newtheorem{Claim}[THEOREM]{Claim}

\newtheorem{Main Claim}[THEOREM]{Main Claim}
\newenvironment{main claim}{\begin{Main Claim}}{\end{Main Claim}}

\newtheorem{Crucial Claim}[THEOREM]{Crucial Claim}
\newenvironment{crucial claim}{\begin{Crucial Claim}}{\end{Crucial Claim}}

\newtheorem{Subclaim}[THEOREM]{Subclaim}

\newtheorem{Sublemma}[THEOREM]{Sublemma}

\newtheorem{Corollary}[THEOREM]{Corollary}

\newtheorem{Example}[THEOREM]{Example}

\newtheorem{Problem}[THEOREM]{Problem}

\newtheorem{Proposition}[THEOREM]{Proposition}

\newtheorem{Conjecture}[THEOREM]{Conjecture}

\newtheorem{Discussion}[THEOREM]{Discussion}

\newenvironment{Proof of the Subfact}
{\noindent{\bf Proof of the Subfact.}}{\par\bigskip}

\newenvironment{Proof of the Theorem}
{\noindent{\bf Proof of the Theorem.}}{\par\bigskip}

\newenvironment{Proof of the Proposition}
{\noindent{\bf Proof of the Proposition.}}{\par\bigskip}

\newenvironment{Proof of the Conclusion}
{\noindent{\bf Proof of the Conclusion.}}{\par\bigskip}

\newenvironment{Proof of the Observation}
{\noindent{\bf Proof of the Observation.}}{\par\bigskip}

\newenvironment{Proof of the Fact}
{\noindent{\bf Proof of the Fact.}}{\par\bigskip}

\newenvironment{Proof of the Lemma}
{\noindent{\bf Proof of the Lemma.}}{\par\bigskip}

\newenvironment{Proof of the Claim}
{\noindent{\bf Proof of the Claim.}}{\par\bigskip}

\newenvironment{Proof of the Corollary}
{\noindent{\bf Proof of the Corollary.}}{\par\bigskip}

\newenvironment{Proof of the Subclaim}
{\noindent{\bf Proof of the Subclaim.}}{\par\medskip}

\newenvironment{Proof of the Main Claim}
{\noindent{\bf Proof of the Main Claim.}}{\par\bigskip}

\newenvironment{Proof of the Crucial Claim}
{\noindent{\bf Proof of the Crucial Claim.}}{\par\bigskip}

\newcommand{\rest}{\upharpoonright}  

\newcommand{\TT}{{\cal T}}

\newcommand{\dom}{\rm dom}

\newcount\skewfactor
\def\mathunderaccent#1#2 {\let\theaccent#1\skewfactor#2
\mathpalette\putaccentunder}
\def\putaccentunder#1#2{\oalign{$#1#2$\crcr\hidewidth
\vbox to.2ex{\hbox{$#1\skew\skewfactor\theaccent{}$}\vss}\hidewidth}}

\makeatletter
\newcommand{\subjclass}[2][2010]{%
  \let\@oldtitle\@title%
  \gdef\@title{\@oldtitle\footnotetext{#1 \emph{Mathematics subject classification.} #2}}%
}
\newcommand{\keywords}[1]{%
  \let\@@oldtitle\@title%
  \gdef\@title{\@@oldtitle\footnotetext{\emph{Key words and phrases.} #1.}}%
}
\makeatother

\author{Mirna D\v zamonja\\ School of Mathematics, University of East Anglia\\Norwich, NR4 7TJ, UK
\\\scriptsize{h020@uea.ac.uk} }

\title{The singular world of singular cardinals}
\subjclass{03E05}
\keywords{singular cardinals, SCH, combinatorics}

\begin{document}
\maketitle
\begin{abstract} The article uses two examples to explore the statement that, contrary to the common wisdom, the properties
of singular cardinals are actually more intuitive than those of the regular ones. 
\footnote{The author thanks EPSRC for their grant EP/I00498. 
}
\end{abstract}

\subjclass{03E05}
\keywords{singular cardinals, SCH, combinatorics}

\section{Introduction} Infinite cardinals can be regular or singular. Regular cardinals and especially successors of regular
cardinals, tend to lend themselves to easier and better understood combinatorial methods and hence are often considered
as being in some sense easier. For example, Todd Eisworth in his Handbook of Set Theory article \cite{todd} artfully exposes difficulties
one has in dealing with the combinatorics of the successors of singulars and explains on a number of examples why the methods used at a successor of a regular most often cannot work when dealing with the successor of a singular. Indeed, it is known that
in many situations, dealing with singular cardinals and their successors has to involve techniques beyond combinatorics and forcing,
and it notably requires large cardinals. This is true even for such seemingly elementary properties as the calculation of the 
size of the power set of the cardinal $\kappa$ as a function of the size of the power sets of the cardinals below, which is
basically the content of the famous Singular Cardinal Hypothesis and which has lead to some of the deepest results throughout set theory. In fact, the common wisdom and the thesis of \cite{todd} are that if the universe is close to $L$ then the singular cardinals
and their successors are ``manageable", and the opposite is true in the models obtained by using strong enough large cardinal
hypothesis.

We shall explore the antithesis, which is that (a) in some situations singular cardinals are more manageable than the regular
ones and (b) in some models obtained from large cardinals the successors of singulars actually behave quite close to how
they do in $L$, even if we do our best to mess up $L$ by changing the cardinal arithmetic drastically. Our exposition will involve two examples, one for each one of (a) and (b), which we now explain.

\section{Tree embeddings} In this part of the paper we discuss an issue that we investigated with V\"a\"an\"anen in \cite{singtrees},
inspired by work that he had initiated in a number of earlier papers, concentrating there on $\aleph_1$ and other
regular cardinals. In contrast, we worked with
$\kappa$ a singular cardinal of countable cofinality and obtained a surprisingly different situation.

Consider rooted trees of height and cardinality $\kappa$, which 
we call $\kappa$-{\em Trees} \footnote{This is different than the usual
{\em $\kappa$-trees}, which
are also required to have levels of size $<\kappa$.}.
We are interested
in $\kappa$-Trees which in addition do not have $\kappa$-branches, which we call {\em bounded}. We studied
the natural notion of reduction,
which is simply a strict-order preserving function from one tree to another. The existence of a reduction from $T$ to $T'$
is denoted by $T\le T'$. Furthermore, we write $T<T'$
if both $T\le T'$ and $T'\not\le T$ hold. 
Considerations of tree reductions have arisen in the context of
model theory and descriptive set theory, as we now explain. Any undefined or unreferenced notion can be found in
 \cite{singtrees}, where in particular we give ample references to EF games and chain models.

A relational structure of size $\lambda$ can be considered as an element of $2^\lambda$. To understand the
classification of such models up to isomorphism one uses the Ehrenfeucht-Fra\"iss\'e  (EF) games, especially in the
countable case.
There is an older method of classification, using
the fact that for two fixed countable models the set of all isomorphisms between them is $F_{\delta\sigma}$, and
hence the set of pairs of models that are non-isomorphic is co-analytic. 
In fact the
set,
\[
\{(A,B):\,A, B\mbox{ countable models  and }\exists f:\,A\cong B\}
\]
is the same set as the set of pairs $(A,B)$ of countable models for which II has a winning strategy in the EF game.
This analysis made it
possible to attach to each  pair $(A,B)$ of non-isomorphic countable models a  rank, called Scott
watershed $S(A,B)$,
which in this case is an ordinal $\alpha<\omega_1$.

The rank can be thought of as a
clock of the EF game in the following sense: during the EF game the Nonisomorphism player I has to 
go down this clock at every stage,
starting at $\alpha$ itself,
and a condition of winning for
I is that he has not run out of time before the lack of an isomorphism has been exposed.
This game is called the dynamic $EF$ game of rank $\alpha+1$ and denoted by $EFD_{\alpha+1}$.
The fact that  $S(A,B)=\alpha+1$ means that
II wins $EFD_\alpha$
(and hence $EFD_\beta$ for any $\beta<\alpha$), while I wins $EFD_{\alpha+1}$
(and hence $EFD_\beta$ for any $\beta>\alpha$).

For uncountable models, say of size $\aleph_1$,
one can generalise the Ehrenfeucht-Fra\"iss\'e Theorem
by considering games of length $\omega_1$. One is  then tempted to find the
corresponding notion of the Scott
watershed.  It turns out that it is no longer enough to use the ordinals,
as the game is transfinite. The right notion this time is that of bounded $\aleph_1$-Trees, the class
of which we shall denote by $\TT_{\aleph_1}$. If $T\le T'$ then it is easier for II to win $EFD_T$ than $EFD_{T'}$.
Respectively, if $T\le T'$ then it is easier for I to win $EFD_{T'}$ than $EFD_{T}$. Finally, if II wins $EFD_{T}$
and I wins $EFD_{T'}$, then one can prove $T<T'$.
In \cite{HyVa} it was shown by Hytinnen and V\"a\"an\"anen 
that the analogue of the Scott watershed exists also in the uncountable case.
Naturally, the importance of $(\TT_{\aleph_1},\le)$ in this context led to a systematic
study of its structural properties and a development of the descriptive set theory
of the space $2^{\omega_1}$ based on $\TT_{\aleph_1}$.  The theory of $(\TT_{\lambda},\le)$ for $\lambda$ successor of
regular is also quite known, although it is not completely parallel to that of
$(\TT_{\aleph_1},\le)$.

When moving to the singular cardinals in \cite{singtrees}, the surprise was that new possibilities opened up. Firstly, it is possible to make links between
chain models and the
infinitary logic $L_{\kappa\kappa}$. For $\kappa$ of singular cofinality one can develop model theory
of $L_{\kappa\kappa}$ based on the concept of a {\em chain model}, as was done by C. Karp and her successors. 
A chain model is an ordinary model $A$
equipped with a presentation of $A$ as a union of a chain $A_0\subseteq A_1\ldots \subseteq A_n\subseteq \ldots$
for $n<\omega$. The chain is not assumed to be elementary. Let us denote such a
system as $(A_n)$. The point of chain models is the following modification of the
truth definition
of $L_{\kappa\kappa}$
\begin{equation}\label{newtruth}
(A_n)\models\exists \bar{x}\varphi( \bar{x})\iff\mbox {there are }n<\omega \mbox{ and } \bar{a}\in A_n \mbox{ with }
A\models \varphi( \bar{a}),
\end{equation}
where $\bar{x}$ is a sequence of length $<\kappa$. If we restrict to chain models, the model
theory of $L_{\kappa\kappa}$ is very much like that of $L_{\omega_1\omega}$.
For example, one can prove the Completeness Theorem using consistency properties, and one can also prove
undefinability of well order, Craig Interpolation Theorem, Beth Definability Theorem, etc.
(None of these theorems is true for the classical $L_{\kappa\kappa}$ logic).
In \cite{singtrees} we extended Scott's analysis of countable models to chain
models of size $\kappa$.
In particular, we considered versions of EF game for (chain) models of a singular cardinality $\kappa$ of countable
cofinality and discovered that the relevant clock trees of these games are bounded
$\kappa$-Trees.

In fact the main point of this is that $\kappa$-Trees for $\kappa$ as above have properties that make
them rather similar to ordinals. The reason for this is that there is a natural notion of rank.
Using this notion we can for example show that the universality number of bounded
$\kappa$-Trees under reduction is just $\kappa^+$, and that
within each rank in $[1,\kappa^+)$ the universality number is just $\omega$. This is in sharp contrast with
the situation of $\lambda$-Trees where $\lambda$ is a regular cardinal. For
example for $\lambda=\aleph_1$ Mekler and V\"a\"an\"anen \cite{MekVa}
have established that
the universality number for bounded $\lambda$-Trees under reduction cannot be computed in ZFC, and the consistency of this number being
equal to 1 for $\lambda=\aleph_1$ is not known.

\section{Universal graphs} In this section we shall discuss an embedding question which comes from an even more familiar object than trees, namely graphs. Given a cardinal $\kappa$, we are interested to know what is the smallest size of a family of graphs of size
$\kappa$ which embeds every graph of size $\kappa$ as an induced subgraph. This is known as the universality number. For
$\kappa=\aleph_0$ this number is 1, as the random graph is a universal countable graph. For uncountable $\kappa$ the situation
becomes sensitive to the axioms of set theory. Namely, by the classical results in model theory on the existence of saturated and special models (see \cite{ChKe}), in the presence of GCH there is a universal graph on every infinite $\kappa$, and in 
fact $\kappa=2^{<\kappa}$ suffices. On the other hand,
a result of Shelah mentioned in \cite{Sh175} and described in \cite{KjSh409}, is that adding $\aleph_2$ Cohen reals to a model
of CH makes
the universality number of graphs at $\aleph_1$ equal to $\aleph_2$. This is an easy proof, in fact Shelah says in the abstract
of \cite{Sh175} ``The consistency of the non-existence of a universal graph of power $\aleph_1$ is trivial. Add $\aleph_2$ generic Cohen reals.", and instead he concentrates on a much more complex proof of the consistency of the existence of a universal graph at $\aleph_1$ and the negation of CH (which in fact was not right in \cite{Sh175}. Shelah  corrected his proof in \cite{Sh175a} and Mekler gave a different proof in \cite{Mekler}). Other successors of regulars behave in a similar way, although neither Mekler's nor Shelah's proof seem to carry over from $\aleph_1$ to larger successors of regulars. Namely, in \cite{DjSh614} D\v zamonja and Shelah 
obtained the consistency of the universality number of graphs at $\kappa^+$ for an arbitrary large regular $\kappa$ being equal 
$\kappa^{++}$ while $2^\kappa$ is as large as desired. The negative consistency results directly translate to other successors of regulars and even to a class of them, and to let the reader appreciate the way Cohen's forcing is used, we give a rendition of that argument here. (The proof of Theorem \ref{Cohensubsets} presented in \cite{KjSh409} is less formal.)

\begin{theorem}[Shelah, see \cite{KjSh409}] \label{Cohensubsets} Suppose that $\kappa^{<\kappa}=\kappa$ and
let $\mathbb P$ be the forcing to add $\lambda$ many, with $\cf(\lambda)\ge\kappa^{++}$ and $\lambda\ge 2^{\kappa^+}$, Cohen subsets to $\kappa$. Then the universality number 
for graphs on $\kappa^+$ in the extension by $\mathbb P$ is $\lambda$.
\end{theorem}

\begin{proof} First notice that in the extension we have that $2^\kappa=2^{\kappa^+}=\lambda$, and hence the total number of graphs on $\kappa^+$
is at most $\lambda$, so the universality number for graphs on $\kappa^+$ is also $\le\lambda$. Now we show that it is $\ge\lambda$.
Suppose to the contrary, that $\{H_\gamma:\,\gamma<\gamma^\ast\}$ for some $\gamma^\ast<\lambda$ in the
extension
are graphs with universe $\kappa^+$ that are universal for graphs on $\kappa^+$. By standard arguments about the factoring of
the Cohen forcing and using $\cf(\lambda)\ge\kappa^{++}$, we may assume that all graphs $H_\gamma$ are from the ground model.
Let $\langle A_i^j:\,i \in [\kappa,\kappa^{+}), j<\kappa^{++}\rangle$ be a 1-1 enumeration of the first $\kappa^{++}$ Cohen subsets 
of $\kappa$ added by $\mathbb P$, where the indexing is chosen for the convenience in the argument to follow. 
For each $j<\kappa^{++}$
we define in the extension a graph $G_j$ on $\kappa^+$ by letting for $\alpha<i<\kappa^+$ there be an
edge between $\alpha$ and $i$ iff $\alpha<\kappa\le i$ and $\alpha\in A^j_i$. For each $j$ let
$h_j$ be an embedding of $G_j$ to some $H_{\gamma_j}$.
Note that there is a club $C$ of $\kappa^{++}$ such that
for all $j\in C$ of cofinality $\kappa^+$, $h_j\rest[\kappa+1)$ is in $V[A^k_i:\,i\in [\kappa,\kappa^+), k<k^\ast]$ for some $k^\ast<j$. Then for any $j\in C$ we have
\begin{equation*}
\begin{split}
A^j_\kappa=\{\alpha<\kappa:\,(\alpha,\kappa)\mbox{ are an edge in } G_j\}=\\\{\alpha<\kappa:\,(h_j\rest[\kappa+1)(\alpha), 
h_j\rest[\kappa+1)(\kappa))\mbox{ are an edge in } H_{\gamma_j}\},
\end{split}
\end{equation*}
which is an object in $V[A^k_i:\,i\in [\kappa,\kappa^+), k<k^\ast]$, a contradiction.
${\eop}_{\ref{Cohensubsets}}$\end{proof}

Using a standard argument about Easton forcing we can see that it is equally easy to get negative universality results for graphs
at a class of regular cardinals:

\begin{theorem}\label{Easton} Suppose that the ground model $V$ satisfies GCH and $\mathcal C$ is a class of regular cardinals in $V$,
while $F$ is a non-decreasing function on $\mathcal C$ satisfying that for each $\kappa\in {\mathcal C}$ we
have $\cf(F(\kappa))\ge\kappa^{++}$. Let $\mathbb P$ be Easton's forcing to add $F(\kappa)$ Cohen subsets to $\kappa$
for each $\kappa\in{\mathcal C }$. Then for each $\kappa\in{\mathcal C }$ the universality number 
for graphs on $\kappa^+$ in the extension by $\mathbb P$ is $F(\kappa)$.
\end{theorem}

\begin{proof} Recall that the forcing notion $\mathbb P$ is constructed as follows: for each $\kappa\in {\mathcal C }$
we have the forcing $P_\kappa$ which adds $F(\kappa)$ Cohen subsets to $\kappa$, using functions 
of size $<\kappa$ from $\kappa\times F(\kappa)$ to $\{0,1\}$
with the extension given by the extension of functions. Then $\mathbb P$ is the Easton product of $\{P_\kappa:\,\kappa\in
{\mathcal C }\}$, which means that each condition $p\in \mathbb P$ is an element of $\Pi_{\kappa\in {\mathcal C}}P_\kappa$
with support ${\rm spt}(p)$ satisfying $|{\rm spt}(p)\cap\theta|<\theta$ for every regular cardinal $\theta$. Denoting by
$p_\kappa$ the projection of condition $p$ on the coordinate $\kappa$, each condition in
${\mathbb P}$ can be viewed as a function on triples $(\kappa, \alpha, \beta)$ where $p(\kappa, \alpha, \beta)$ is
defined as $p_\kappa(\alpha,\beta)$. Then standard arguments show that for every regular $\theta$ the forcing
breaks into ${\mathbb P}^{\le\theta}\times {\mathbb P}^{>\theta}$ where ${\mathbb P}^{\le\theta}=\{p\rest (\kappa,\alpha,\beta):\,\kappa\le
\theta\}$ and ${\mathbb P}^{>\theta}=\{p\rest (\kappa,\alpha,\beta):\,\kappa>
\theta\}$, and that ${\mathbb P}^{>\theta}$ is $\theta$-closed while ${\mathbb P}^{\le\theta}$ satisfies the $\theta^+$-chain condition. 

Now let $\kappa\in \mathcal C$ and suppose for a contradiction that in the extension by $\mathbb P$ we have 
a universal family $\{H_\gamma:\,\gamma<\gamma^\ast\}$ of graphs on $\kappa^+$ for some $\gamma^\ast<F(\kappa)$.
Since ${\mathbb P}^{>\kappa^+}$ is $\kappa^+$-closed, it does not add any new subsets to $\kappa^+$, and hence
the universal family is added by ${\mathbb P}^{\le{\kappa^+}}$. We shall once more use an argument about factoring, in that
for every $\theta<F(\kappa)$ we can consider ${\mathbb P}^{\le{\kappa^+}}$ as the product
\begin{equation*}
\begin{split}
Q^{\le \theta}=\{p\in {\mathbb P}^{\le{\kappa^+}}:\,(\kappa,\alpha,\beta)\in \dom(p)\implies \beta\le \theta\}\\
\times\quad Q^{>\theta} =\{p\in {\mathbb P}^{\le{\kappa^+}}:\,(\kappa,\alpha,\beta)\in \dom(p)\implies \beta> \theta\}.
\end{split}
\end{equation*}
Since $\cf(F(\kappa))\ge \kappa^{++}$,
there is some $\theta<F(\kappa)$ such that all graphs $H_\gamma$ are added by $Q^{\le \theta}$. Now we can basically
repeat the argument from the proof of Theorem \ref{Cohensubsets}: let $\langle A_i^j:\,i \in [\kappa,\kappa^{+}), j<F(\kappa)\rangle$ be a 1-1 enumeration of the Cohen subsets 
of $\kappa$ added by $Q^{>\theta} $ and for
each $j<\kappa^{++}$
we define in the extension a graph $G_j$ on $\kappa^+$ by letting for $\alpha<i<\kappa^+$ there be an
edge between $\alpha$ and $i$ iff $\alpha<\kappa\le i$ and $\alpha\in A^j_i$. For each $j$ let
$h_j$ be an embedding of $G_j$ to some $H_{\gamma_j}$.
Note that there is a club $C$ of $F(\kappa)$ such that
for all $j\in C$ of cofinality $\ge \kappa^+$, $h_j\rest[\kappa+1)$ is in $V[A^k_i:\,i\in [\kappa,\kappa^+), k<k^\ast]$ for some $k^\ast<j$. Then for every $j\in C$
\begin{equation*}
\begin{split}
A^j_\kappa=\{\alpha<\kappa:\,(\alpha,\kappa)\mbox{ are an edge in} G_j\}=\\\{\alpha<\kappa:\,(h_j\rest[\kappa+1)(\alpha), 
h_j\rest[\kappa+1)(\kappa))\mbox{ are an edge in } H_{\gamma_j}\},
\end{split}
\end{equation*}
which is an object in $V^{Q^{\le \theta}}[A^k_i:\,i\in [\kappa,\kappa^+), k<k^\ast]$, a contradiction.
${\eop}_{\ref{Easton}}$
\end{proof}

Things change at the successor of a singular! Positive results analogous to the D\v zamonja-Shelah  \cite{DjSh614} were obtained
for $\kappa$ of countable cofinality by D\v zamonja and Shelah in \cite{DzShPrikry} and for arbitrary cofinality by 
Cummings, D\v zamonja, Magidor, Morgan and Shelah in \cite{5authorsforcing}. We quote that general result:

\begin{theorem}[Cummings et al. \cite{5authorsforcing}] \label{glavni} If $\kappa$ is a supercompact cardinal,
$\lambda<\kappa$ is a regular cardinal and 
$\Theta$ is a cardinal with $\cf(\Theta)\ge\kappa^{++}$ and $\kappa^{+3}\le\Theta$
there is a forcing extension in which $\cf(\kappa)=\lambda$,  $2^{\kappa}=2^{\kappa^+} =
\Theta$ and there is a universal family of graphs on $\kappa^+$ of size $\kappa^{++}$.
\end{theorem}

However, it is not known in Theorem \ref{glavni} if the universality number is exactly $\kappa^{++}$. For all we know, in that
model there could be a universal graph on $\kappa^+$. Worse, we do not know how to imitate the negative universality
presented in Theorems \ref{Cohensubsets} and \ref{Easton} above. We do not know how to obtain a model in which the relevant instances of GCH fail and the universality number
of graphs is $2^\lambda$ for $\lambda$ the successor of a singular. We have given Shelah's argument about the Cohen forcing in 
gory detail to invite the reader to think if anything like this can be produced at a singular $\kappa$. Initial results by Cummings
and Magidor (private communication) indicate that a naive generalization might not be possible. Perhaps there is some sort of singular cardinal hypothesis-like behaviour here. Perhaps we cannot just monkey around with the universality number for
graphs at the successor of  singular even when we are basically as far from $L$ as we can possibly be, at least as far as the power set function is concerned? 

\section{Conclusion} We have discussed two problems where the intuition of the singular cardinal being in a sense more difficult than a regular one, seems to be completely false. In fact, the truth seems to be that although the properties of the singular cardinals are
{\em harder to discover}, once we have done that difficult discovery, these properties are actually {\em nicer} than their analogues at the regular
cardinals. Some other results but the ones presented here can be viewed with this idea in mind, for example does not the whole
story of the Singular Cardinal Hypothesis including  the celebrated theorem of Shelah
$[(\forall n< \omega)  2^{\aleph_n}<\aleph_\omega]\implies 2^{\aleph_\omega}<\aleph_{\omega_4}$, does not this story say that
the singulars are in fact more intuitive than the regulars? Erd\"os has said something to the extent of the infinite being the
easy part, and the finite the difficult one. If the infinite is the limit of the finite, a singular cardinal is a limit of the successors of regulars,
and maybe it is at such limits that the unruly universe of set theory wishes to express its more tame behaviour. It seems possible
that by investigating finer combinatorics than that expressed by the power set function we may find combinatorial versions
of SCH which are just outright true.

\bibliographystyle{plain}
\bibliography{biblio}

\end{document}